\documentclass[12pt]{amsart}
\usepackage[T1]{fontenc}
\usepackage{amssymb}
\usepackage{bm}
\newtheorem{thm}{Theorem}
\newtheorem{lem}{Lemma}[section]
\newtheorem{prop}{Proposition}[section]

\newtheorem{defi}{Definition}[section]
\numberwithin{equation}{section}
\newcommand{\N}{\mathbb{N}}

\begin{document}

\title[
purely exponential equations with four terms
]
{
Solving an infinite number of purely exponential Diophantine equations \\with four terms}

\author{Takafumi Miyazaki}
\address{Takafumi Miyazaki
\hfill\break\indent Gunma University, Division of Pure and Applied Science,
\hfill\break\indent Faculty of Science and Technology
\hfill\break\indent Tenjin-cho 1-5-1, Kiryu 376-8515.
\hfill\break\indent Japan}
\email{tmiyazaki@gunma-u.ac.jp}
\thanks{The author is supported by JSPS KAKENHI (No. 24K06642).}

\subjclass[2010]{11D61, 11J86, 11J61, 11D79}
\keywords{$S$-unit equation, 
purely exponential equation, Baker's method, non-Archimedean valuation, Skolem's conjecture}

\maketitle

\markleft{Takafumi Miyazaki}
\markright{purely exponential equations with four terms}

\begin{abstract}
An important unsolved problem in Diophantine number theory is to establish a general method to effectively find all solutions to any given $S$-unit equation with at least four terms.
Although there are many works contributing to this problem in literature, most of which handle purely exponential Diophantine equations, it can be said that all of them only solve finitely many equations in a natural distinction.
In this paper, we study infinitely many purely exponential Diophantine equations with four terms of consecutive bases. 
Our result states that all solutions to the equation $n^x+(n+1)^y+(n+2)^z=(n+3)^w$ in positive integers $n,x,y,z,w$ with $n \equiv 3 \pmod{4}$ are given by $(n,x,y,z,w)=(3,3,1,1,2), (3,3,3,3,3)$.
The proof uses elementary congruence arguments developed in the study of ternary case, Baker's method in both rational and $p$-adic cases, and the algorithm of Bert\'ok and Hajdu based on a variant of Skolem's conjecture on purely exponential equations. 
\end{abstract}

\maketitle

\section{Introduction}

$S$-unit equations are an important object in Diophantine number theory.
One of the well-known unsolved problems on them is to establish a general method to effectively find all solutions to any given $S$-unit equation with at least four terms.
In this paper, we contribute to this subject by studying purely exponential Diophantine equations with several terms, which are typical examples of such $S$-nuit equations.

Every $S$-unit equation over the rationals can be expressed by the following purely exponential Diophantine equation:
\begin{equation}\label{pexp}
c_1 {a_{11}}^{x_{11}}\cdots{a_{1\hspace{0.01cm}l_1}}^{x_{1\hspace{0.01cm}l_1}}+
c_2 {a_{21}}^{x_{21}}\cdots{a_{2\hspace{0.01cm}l_2}}^{x_{2\hspace{0.01cm}l_2}}+
\cdots+
c_k {a_{k1}}^{x_{k1}}\cdots{a_{k\hspace{0.01cm}l_k}}^{x_{k\hspace{0.01cm}l_k}}
=0
\end{equation}
with $k \ge 3$ and $l_1 \ge 0, l_2 \ge 1, \ldots, l_k \ge 1$, where each letter using $a,c$ and $x$ denotes a fixed positive integer greater than 1, a fixed nonzero integer, and an unknown positive integer, respectively.
It is well-known that if the number of terms $k$ of \eqref{pexp} equals 3, then the theory of linear forms in logarithms established by Baker gives us an upper bound for all unknown exponents $x_{11},\ldots,x_{3 \hspace{0.01cm}l_3}$ which depends only on the base numbers $a_{11},\ldots,a_{3 \hspace{0.01cm}l_3}$ and coefficients $c_1,c_2,c_3$ and is effectively computable.
While if $k>3$, there is no still found any general method to give such bounds for the unknown exponents, which corresponds to the unsolved problem mentioned at the beginning of this paper.
Below, we will consider the problem of determining the solutions to \eqref{pexp} with several terms, namely, for the case where $k>3$.
For cases which can be reduced to ternary case, or topics on the number of solutions to $S$-unit equations including \eqref{pexp}, see recent papers \cite{MiPi,MiPi2,MiPi3,ScoSt} and monographs \cite[Ch.1]{ShTi} and \cite[Ch.s 4--6]{EvGy}.

A classical well-known work in the direction of several terms is one by Brenner and Foster \cite{BrFo}, who gave a variety of results to solve equations \eqref{pexp} with $k>3$ for small values of the base numbers.
Their main idea is to find congruence restrictions on unknown exponents by reducing a given equation with several moduli.
Similar techniques are used in most of the other related works in the same period (cf.~\cite{AlFo,AlFo2}).
Although the mentioned method using moduli can be used in general, it should be remarked that those moduli have been chosen `accordingly'. 
We mention that Skinner \cite{Sk} gave another treatment of a certain class of equations \eqref{pexp} with the base numbers sharing large factors by using Baker's method in both rational and $p$-adic cases.
Related to this work, Bajpai and Bennett \cite{BaBen} recently showed a general effective treatment of 5-term $S$-unit equations over any number field with the size of $S$, the number of involved places,  extremely small.

On the other hand, Bert\'ok and Hajdu \cite{BerHa} in 2016 presented a systematic treatment on choosing the mentioned moduli. 
Their algorithm is based on a variant of Skolem's conjecture, which roughly asserts that if a purely exponential Diophantine equation has no solution, then it has no solution already modulo some integer (cf.~\cite[pp. 398--399]{Sch}).
Indeed, they succeeded in solving many equations \eqref{pexp} which have been beyond existing methods.
Very recently Dimitrov and Howe \cite{DiHo} used the algorithm to study problems including Erd{\H o}s' conjecture on the representations of powers of $2$ by sum of distinct powers of $3$.
For other applications, see \cite{Ber,MiTe,BerHa2}. 

Although there have been a number of works to determine the solutions to equation \eqref{pexp} as seen above, it can be said that all of them only handle finitely many equations under a natural distinction.
The motivation of this paper is to study an infinite family of equations \eqref{pexp} with several terms in that sense, and solve all in it completely.
We are in position to state the result of this paper.

\begin{thm}\label{th-3456}
Let $n$ be a positive integer with $n \equiv 3 \pmod{4}.$
Then the equation
\[ n^x+(n+1)^y+(n+2)^z=(n+3)^w \]
has no solution in positive integers $x,y,z$ and $w,$ except for $n=3,$ where all solutions are given by $(x,y,z,w)=(3,1,1,2),(3,3,3,3).$
\end{thm}

This theorem handles purely exponential Diophantine equations with four terms of consecutive bases, and it is the first result which essentially solves infinitely many equations \eqref{pexp} with several terms.
It might be worth noting that this situation can be regarded as an analogue to the history with $k=3$ that the classical well-known work of Sierpi\'nski \cite{Si}, which solved the equation $3^x+4^y=5^z$, was generalized to the equation $n^x+(n+1)^y=(n+2)^z$ with fixed positive integers $n$ (see \cite{HeTo} and the references therein).

The present paper is organized as follows. 
In the next section, we prepare a number of elementary lemmas giving restrictions on possible solutions of the equation in our theorem.
In particular, we show that its bases numbers have to be of particular forms, and we use methods developed in the study of ternary case to derive a sharp lower bound for a certain quantity composed of some of the unknown exponents.
In Section \ref{sec-3}, we show in the equation that the sizes of the first and second terms of the left-hand side have to be much smaller than that of the term of the right-hand one.
For this purpose, we rely on Baker's method in $p$-adic form with some particular primes $p$.
It turns out that the size of the third term of the left-hand side is very close to that of the term of the right-hand side.
In Section \ref{sec-4}, using this fact together with a usual application of Baker's method in rational case, we obtain a relatively sharp upper bound for all of the exponential unknowns. 
We should mention that our applications of Baker's bounds for linear forms in logarithms presented in Sections \ref{sec-3} and \ref{sec-4} are similar ones appearing in \cite{Sk} and \cite{BaBen}.
The combination of the consequences of Sections \ref{sec-2}, \ref{sec-3} and \ref{sec-4} enables us to finish the proof of our theorem, except for considering several concrete equations. 
The final section is devoted to handle such equations. 
For this, we adopt using the algorithm of Bert\'ok and Hajdu mentioned before, where we mainly explain how their algorithm works well for a particular equation, $3^x+4^y+5^z=6^w$.

\section{Elementary lemmas}\label{sec-2}

Let $n$ be a positive integer.
Write $N=n+1$ in the equation under consideration.
We have
\begin{equation}\label{eq-N}
(N-1)^x+N^y+(N+1)^z=(N+2)^w
\end{equation}
in positive integers $x,y,z$ and $w$.
Below, we prepare many lemmas.

\begin{lem}\label{lem-mod3}
Equation \eqref{eq-N} has no solution if $N \equiv 2 \pmod{3}.$
\end{lem}

\begin{proof}
Suppose that $N \equiv 2 \pmod{3}$.
Reducing equation \eqref{eq-N} modulo $3$ gives that $1+2^y \equiv 4^w \pmod{3}$, leading to $2^y \equiv 0 \pmod{3}$.
This contradiction shows the assertion.
\end{proof}

By Lemma \ref{lem-mod3}, we may assume that 
\begin{equation}\label{cond-N-mod3}
N \not\equiv 2 \mod{3}. 
\end{equation}

In what follows, let $(x,y,z,w)$ be a solution to \eqref{eq-N}.

\begin{lem}\label{lem-modN}
The following hold.
\[
\begin{cases}
\,2^w \equiv 0 \pmod{N} & \text{if $x$ is odd},\\
\,2^w \equiv 2 \pmod{N} & \text{if $x$ is even}.\\
\end{cases}
\]
\end{lem}

\begin{proof}
Reducing equation \eqref{eq-N} modulo $N$ gives that
\[
(-1)^x+1 \equiv 2^w \mod{N}.
\]
This implies the assertion.
\end{proof}

From now on, assume that $n \equiv 3 \pmod{4}$, i.e.,
\begin{equation}\label{ass-N}
N \equiv 0 \mod{4}.
\end{equation}

\begin{lem}\label{lem-xodd}
$x$ is odd.
\end{lem}

\begin{proof}
Suppose on the contrary that $x$ is even.
Lemma \ref{lem-modN} says that
\[
2^w \equiv 2 \mod{N}.
\]
Since $4 \mid N$ by \eqref{ass-N}, one has
\[
2^w \equiv 2 \mod{4}.
\]
This implies that $w=1$.
However, equation \eqref{eq-N} for this case clearly does not hold.
\end{proof}

Lemmas \ref{lem-modN} and \ref{lem-xodd} together say that 
\[
2^w \equiv 0 \mod{N}.
\]
This means that $N$ has to be a power of $2$.
Further, this fact together with condition \eqref{cond-N-mod3} tells us that $N \equiv 1 \pmod{3}$. 
To sum up, we can write
\[
N=4^e
\]
for some $e \in \N$. 
Then equation \eqref{eq-N} becomes
\begin{equation}\label{eq-e}
(4^e-1)^x+4^{ey}+(4^e+1)^z=(4^e+2)^w
\end{equation}
with 
\[
N=4^e.
\]

We remark in \eqref{eq-e} that the following inequalities hold:
\[
x<\mu_x w, \ \ y<\mu_y w, \ \ z<\mu_z w, 
\]
where 
\[
\mu_x=\frac{\log (N+2)}{\log(N-1)}, \ \ \mu_y=\frac{\log (N+2)}{\log N}, \ \ \mu_z=\frac{\log (N+2)}{\log(N+1)}.
\]
Further, observe that
\begin{gather}\label{cong-mod2mod3}
N-1 \equiv 0 \pmod{3}, \quad N+2 \equiv 0 \pmod{6}.
\end{gather}

\begin{lem}\label{lem-yodd-x1mod2e}
The following hold.
\begin{itemize}
\item[(i)] $y$ is odd.
\item[(ii)] $x \equiv 1 \pmod{2e}$ with $x>1.$
In particular, $x \ge 2e+1.$
\end{itemize}
\end{lem}

\begin{proof}
Reducing equation \eqref{eq-N} modulo $N+1$ gives that
\[
(-2)^x+(-1)^y \equiv 1 \mod{N+1}.
\]
From this it is easy to see that $y$ is odd as $N+1$ is odd.
Further, noting that $x$ is odd by Lemma \ref{lem-xodd}, one finds that  the above displayed congruence becomes $2^x \equiv -2 \pmod{N+1}$ with $N=2^{2e}$, so that 
\[
2^{x-1} \equiv -1 \mod{2^{2e}+1}.
\]
This implies that $x>1$, and also that $x-1 \equiv 0 \pmod{2e}$ as $2e$ is the least one among the positive integers $t$ such that $2^t \equiv -1 \pmod{2^{2e}+1}$. 
\end{proof}

\begin{lem}\label{lem-wge2e1}
$w \ge 2e+1$ if $(e,w) \ne (1,2).$
\end{lem}

\begin{proof}
We know that
\[
w > \frac{\log(4^e-1)}{\log (4^e+2)}\,x.
\]
Since $x \ge 2e+1$ by Lemma \ref{lem-yodd-x1mod2e}\,(ii), the above inequality implies the assertion.
\end{proof}

If $(e,w)=(1,2),$ then $(x,y,z)=(3,1,1)$.
Thus, by Lemma \ref{lem-wge2e1}, we may assume that
\[
w \ge 2e+1.
\]

\begin{lem}\label{lem-zodd}
$z$ is odd.
\end{lem}

\begin{proof}
Reducing equation \eqref{eq-e} modulo 3 gives that $1+2^z \equiv 0 \pmod{3}$. 
This implies the assertion.
\end{proof}

\begin{lem}\label{lem-mod4e}
$y>1$ and $x+z \equiv 2^{w-2e} \pmod{4^e}.$
\end{lem}

\begin{proof}
First, recall that both $x,z$ are odd, and that $w \ge 2e+1$.
We expand three terms in equation \eqref{eq-e} as follows:
\begin{align*}
&(4^e-1)^x = -1 + 4^e x - 4^{2e} \binom{x}{2} + 4^{3e} \binom{x}{3} - \cdots\,,
\\
&(4^e+1)^z=1 + 4^e z + 4^{2e} \binom{z}{2} + 4^{3e} \binom{z}{3} + \cdots\,,
\\
&(2+4^e)^w=2^w + w\,2^{w-1}4^e  +\binom{w}{2} 2^{w-2}4^{2e}  + \binom{w}{3}2^{w-3}4^{3e}  + \cdots\,.
\end{align*}
Substituting these into equation \eqref{eq-e}, and dividing both sides of the resulting one by $4^e$, one finds that
\begin{equation}\label{eq-e-2}
\begin{split}
& \ \ \ \ x - 4^{e} \binom{x}{2} + 4^{2e} \binom{x}{3} - \cdots
\\
&+4^{e(y-1)}
+z + 4^{e} \binom{z}{2} + 4^{2e} \binom{z}{3} + \cdots
\\
= & \ 2^{w-2e} +  w\,2^{w-1} + \binom{w}{2} 2^{w-2}4^{e} + \binom{w}{3} 2^{w-3}4^{2e} + \cdots\,.
\end{split}
\end{equation}
Reducing this modulo $4^e$ implies that
\[
x+(4^e)^{y-1}+z \equiv 2^{w-2e} \mod{4^e}.
\]
Since the left-hand side has to be even, it follows that $(4^e)^{y-1}$ is even, so that $y>1$, and $x+z \equiv 2^{w-2e} \pmod{4^e}$.
\end{proof}

\begin{lem}\label{lem-mod4e-2}
The following hold.
\[
\begin{cases}
\,x+z \equiv 0 \pmod{4^e} & \text{if $w \ge 4e$},\\
\,x+z=2^{w-2e} & \text{if $w<4e$ and $(e,w) \ne (1,3)$}.\\
\end{cases}
\]
\end{lem}

\begin{proof}
By Lemma \ref{lem-mod4e}, it suffices to consider when $w<4e$, and we have
\[
x+z \equiv 2^{w-2e} \mod{4^e}.
\]
This implies that either $x+z=2^{w-2e}$ or $x+z \ge 2^{w-2e}+4^e$. 
Suppose that the latter case holds. 
Since $x<\mu_x w$ and $z<\mu_z w$, one has
\[
2^{w-2e}+4^e < (\mu_x+\mu_z)\, w.
\]
This together with the inequality $2e+1 \le w<4e$ implies that $e=1$ and $w=3$.
This completes the proof.
\end{proof}

If $(e,w)=(1,3)$, then $(x,y,z)=(3,3,3)$.
Thus, by Lemma \ref{lem-mod4e-2}, we may assume that 
\[
x+z=2^{w-2e},
\]
whenever $w<4e$.

\begin{lem}\label{lem-wge4e}
$w \ge 4e$ if $(e,x,y,z,w) \ne (1,3,3,3,3).$
\end{lem}

\begin{proof}
Suppose that $w<4e$.
Then
\begin{equation}\label{eq-nontriv-1}
x+z=2^{w-2e}.
\end{equation}
We will observe that this leads to $(e,x,y,z,w)=(1,3,3,3,3)$.

First, we consider the case where $w \ge 3e$.
Since $x+z<(\mu_x+\mu_z)\,w$, it follows from \eqref{eq-nontriv-1} that
\[
2^{w-2e}<(\mu_x+\mu_z)\,w.
\]
This together with the inequality $3e \le w<4e$ implies that $e \le 4$ and $w \le 12$.
A brute force method on equation \eqref{eq-e} with \eqref{eq-nontriv-1} shows that $(e,x,y,z,w)=(1,3,3,3,3)$.
Thus, in what follows, we may assume that
\begin{equation}\label{ineq-w-small}
w<3e
\end{equation}
with $e \ge 5$.
We will observe that this leads to a contradiction.

As seen in the proof of Lemma \ref{lem-mod4e}, we can rewrite equation \eqref{eq-e} as \eqref{eq-e-2}.
Inserting relation \eqref{eq-nontriv-1} into \eqref{eq-e-2}, and dividing both sides of the resulting equation by $2^{2e-1}$, one finds that
\begin{align*}
& \ \ - 2 \binom{x}{2} +2\cdot 4^{e} \binom{x}{3} - \cdots
\\
&+2  \cdot 4^{e(y-2)}
+ 2  \cdot \binom{z}{2} + 2 \cdot 4^{e}\binom{z}{3} + \cdots
\\
=\,&\, w\,2^{w-2e} + \binom{w}{2} 2^{w-1} + \binom{w}{3} 2^{w-2}4^{e} + \cdots\,.
\end{align*}
Further, since $y \ge 3$ by Lemmas \ref{lem-yodd-x1mod2e}\,(i) and \ref{lem-mod4e}, and $w \ge 2e+1$, reducing the above equation modulo $4^e$ leads to
\begin{equation}\label{cong-xzw}
-x(x-1)+z(z-1) \equiv w\,2^{w-2e} \mod{4^e}.
\end{equation}
We shall estimate the sizes of the terms appearing in this congruence.
Observe from \eqref{ineq-w-small} that
\begin{align*}
\bigr|-x(x-1)+z(z-1)\bigr|
&=|z-x|\,(z+x-1)\\
&<(\mu_x w-1)\,\big(\,(\mu_x+\mu_z)\,w-1\,\big)\\ 
&\le \big(\,\mu_x (3e-1)-1\,\big)\,\big(\,(\mu_x+\mu_z)(3e-1)-1\,\big)
<4^e/2,
\end{align*}
\[
w\,2^{w-2e} \le (3e-1)\,2^{e-1}<4^e/2.
\]
To sum up, congruence \eqref{cong-xzw} is an equality, namely,
\[
-x(x-1)+z(z-1)=2^{w-2e}\,w.
\]
Since $2^{w-2e}=z+x$ by \eqref{eq-nontriv-1}, one has
\[
(z-x)(z+x-1)=(z+x)\,w.
\]
This implies that $z+x$ divides $z-x$, which holds only if $z-x=0$, so that $w=0$, a contradiction.
\end{proof}

By Lemmas \ref{lem-mod4e-2} and \ref{lem-wge4e}, we may assume that
\[
x+z \equiv 0 \mod{N}
\]
with $N=4^e$.
In particular, this implies that $x+z \ge N$, so that
\begin{equation}\label{ineq-w-low}
w>\frac{1}{\mu_x+\mu_z}\,N.
\end{equation}

\section{Applications of Baker's method in $p$-adic case}\label{sec-3}

Recall from congruences \eqref{cong-mod2mod3} that in equation \eqref{eq-e} both terms $4^{ey}$ and $(4^e+2)^w$ are even, and both terms $(4^e-1)^x$ and $(4^e+2)^w$ are divisible by 3.
Based on these conditions, here we prove the following:

\begin{lem}\label{lem-padic}
If $e \ge 9,$ then the following hold.
\begin{itemize}
\item[\rm (i)] $y<w/2.$
\item[\rm (ii)] $x<w/2.$
\end{itemize}
\end{lem}

\begin{defi}\rm
For any algebraic number $\alpha$ of degree $d$ over $\mathbb Q$, we define the absolute logarithmic height ${\rm h}(\alpha)$ of $\alpha$ by the formula
\[
{\rm h}(\alpha)=\frac{1}{d}\,
\Big(\log |c_0| + \sum \log\,\max\{\,1,|\alpha'|\,\}\,\Big),
\]
where $c_0$ is the leading coefficient of the minimal polynomial of $\alpha$ over $\mathbb Z$, and the sum extends over all conjugates $\alpha'$ of $\alpha$ in the field of complex numbers.
\end{defi}

To use Baker's method in $p$-adic form, we rely on the following well-known result of Bugeaud which is a particular case of \cite[Theorem 2.13]{Bu-book}.

\begin{prop}\label{prop-padic}
Let $p$ be a prime.
Let $\alpha_1$ and $\alpha_2$ be nonzero rational numbers such that $\nu_{p}(\alpha_1)=0$ and $\nu_{p}(\alpha_2)=0.$
Assume that $\alpha_1$ and $\alpha_2$ are multiplicatively independent.
Let ${\rm g}$ be a positive integer such that
\begin{gather*} \label{padic-ass1}
\nu_{p}( {\alpha_1}^{{\rm g}}-1 ) \ge E, \ \
\nu_{p}( {\alpha_2}^{{\rm g}}-1 ) \ge E
\end{gather*}
for some positive number $E$ with $E>1+\frac{1}{p-1}.$
If $p=2,$ then further assume that
\[
\nu_{2}( \alpha_2-1 ) \ge 2.
\]
Let $H_1$ and $H_2$ be positive numbers such that
\[
H_j \ge \max \{ \,{\rm h}(\alpha_j),\,E\log p\, \}, \quad j=1,2.
\]
Then, for any positive integers $b_1$ and $b_2,$
\[
\nu_{p}( {\alpha_1}^{b_1} - {\alpha_2}^{b_2}) \le \frac{36.1 \, {\rm g}\, H_1 H_2}{E^3\log^4 p}\,\Bigr(\max \{\, \log b'+\log (E\log p)+0.4,\,6E\log p \,\}\Bigr)^2
\]
with $b'=b_1/H_2+b_2/H_1.$
\end{prop}

\begin{proof}[Proof of Lemma $\ref{lem-padic}$]
Assume that $e>1$. \par
(i) Since $x$ is odd by Lemma \ref{lem-xodd}, reducing equation \eqref{eq-e} modulo $2^{\min\{y,w\}}$ shows
\begin{equation}\label{2adic-ineq1}
\min\{y,w\} \le \nu_{2}\bigr(\, (4^e+1)^z - (1-4^e)^x \,\bigr).
\end{equation}
To find an upper bound for the right-had side above, we use Proposition \ref{prop-padic} with the following parameters:
\begin{align*}
&p=2, \ \alpha_1=4^e+1, \ \alpha_2=1-4^e, \ b_1=z, \ b_2=x, \\ 
&{\rm g}=1, \ E=2e, \ H_1=\log (4^e+1), \ H_2=2e\log 2.
\end{align*}
It follows that
\begin{equation}\label{2adic-ineq2}
\nu_{2}\bigr(\, (4^e+1)^z - (1-4^e)^x \,\bigr) \le \mathcal C\,\Bigr(\max \bigr\{\log b'+\log (2e\log 2)+0.4,\,12e\log 2\,\bigr\}\Bigr)^2,
\end{equation}
where
\[
\mathcal C=\mathcal C(e)=\frac{36.1 \, \log (4^e+1)}{(2e)^2\log^3 2},\ \ \ b'=\frac{z}{2e\log 2}+\frac{x}{\log (4^e+1)}.
\]
Observe that
\[
b' \cdot 2e\log 2 \cdot \exp(0.4)<(z+x) \cdot \exp(0.4)<(\mu_z+\mu_x)\,w \cdot  \exp(0.4)<3w.
\]

Now, suppose that $y \ge w/2$.
We will observe that this leads to $e \le 8$.
One uses inequalities \eqref{2adic-ineq1} and \eqref{2adic-ineq2} together to find that
\[
w<2\,\mathcal C\,\bigr(\max \{\log (3w),\,12e\log 2\}\,\bigr)^2.
\]
If $\log (3w)>12e\log 2$, then $w>2^{12e}/3$, and
\[
\frac{w}{\log^2 (3w)} < 2\,\mathcal C=\frac{2 \cdot 36.1 \, \log (4^e+1)}{(2e)^2\log^3 2}.
\]
It is easy to see that these inequalities are compatible only if $e \le 8$.
While if $\log (3w) \le 12e\log 2$, then  
\[
w<2\,\mathcal C\,(12e\log 2)^2.
\]
This together with inequality \eqref{ineq-w-low} gives
\[
\frac{4^e}{\mu_x+\mu_z} <\frac{72 \cdot 36.1 \, \log (4^e+1)}{\log 2}.
\]
This implies that $e \le 8$.
%
%
\par
(ii) The proof proceeds along similar lines to that of (i). 
Since $z$ is odd by Lemma \ref{lem-zodd}, reducing equation \eqref{eq-e} modulo $2^{\min\{x,w\}}$ shows
\begin{equation}\label{3adic-ineq1}
\min\{x,w\} \le \nu_{3}\bigr(\, (-4^e-1)^z-(-2)^{2ey} \,\bigr).
\end{equation}
To find an upper bound for the right-hand side above, we use Proposition \ref{prop-padic} with the following parameters:
\begin{align*}
&p=3, \ \alpha_1=-4^e-1, \ \alpha_2=-2, \ b_1=z, \ b_2=2ey, \\ 
&{\rm g}=3, \ E=2, \ H_1=\log (4^e+1), \ H_2=2\log 3.
\end{align*}
It follows that
\begin{equation}\label{3adic-ineq2}
\nu_{3}\bigr(\, (-4^e-1)^z-(-2)^{2ey} \,\bigr) \le \mathcal C\, \Bigr(\max \bigr\{\log b'+\log (2\log 3)+0.4,\,12\log 3\,\bigr\}\,\Bigr)^2,
\end{equation}
where
\[
\mathcal C=\mathcal C(e)=\frac{36.1\cdot{\rm 3}\, \log (4^e+1)}{2^2\log^3 3},\ \ \ b'=\frac{z}{2\log 3}+\frac{2ey}{\log (4^e+1)}.
\]
Observe that
\[
b' \cdot 2\log 3 \cdot \exp(0.4)<(\,z+(2\log 3)y\,) \cdot \exp(0.4)< \exp(0.4)\cdot  (\mu_z+2\log 3\,\mu_y)\,w< 5w.
\]
Now, suppose that $x \ge w/2$.
Then inequalities \eqref{3adic-ineq1} and \eqref{3adic-ineq2} together shows that
\[
w<2\,\mathcal C\,\bigr(\max \{\log (5w),\,12\log 3\}\,\bigr)^2.
\]
Similarly to the proof of (i), the above inequality together with inequality \eqref{ineq-w-low} implies that $e \le 8$.
\end{proof}

\section{Application of Baker's method in rational case}\label{sec-4}

In what follows, we also write
\[
a=N-1, \ b=N, \ c=N+1, \ d=N+2.
\]
In this section, we assume the conclusion of Lemma \ref{lem-padic}, namely,
\begin{equation}\label{ineq-xy-upp}
\max\{x,y\}<\frac{w}{2}.
\end{equation}
Then, it is not hard to see that $c^z>\max\{a^x,b^y\}$ in equation \eqref{eq-e}, so that
\begin{equation}\label{ineq-czapproxdw}
c^z>\frac{d^w}{3}.
\end{equation}

Dividing both sides of equation \eqref{eq-e} by $c^z$ gives
\[
\frac{a^x+b^y}{c^z}+1=\frac{d^w}{c^z}.
\]
Since $\max\{a,b\}<d$, it follows from inequalities \eqref{ineq-xy-upp} and \eqref{ineq-czapproxdw} that
\[
\frac{a^x+b^y}{c^z}<\frac{d^x+d^y}{c^z}<\frac{\,2d^{w/2}\,}{d^w/3}=\frac{6}{d^{w/2}}.
\]
Then
\[
\log \frac{\,d^w\,}{c^z}=\log\biggr(\frac{a^x+b^y}{c^z}+1\biggr)<\frac{a^x+b^y}{c^z}<\frac{6}{d^{w/2}}.
\]
Therefore,
\begin{equation}\label{ineq-loglam-upp}
\log \varLambda<-\frac{\log d}{2}\,w+\log 6,
\end{equation}
where 
\[
\ \ \ \varLambda=w \log d - z \log c \ \ (>0).
\]

To obtain a lower bound for the left-hand side of \eqref{ineq-loglam-upp}, we rely on the following result of Laurent which is a particular case of \cite[Corollary 2;\,$(m,C_2)=(10,25.2)$]{La}:

\begin{prop}\label{prop-comp}
Let $\alpha_1$ and $\alpha_2$ be positive rational numbers with both greater than $1.$
Assume that $\alpha_1$ and $\alpha_2$ are multiplicatively independent.
Let $H_1$ and $H_2$ be positive numbers such that
\[
H_j \ge \max \{ \,{\rm h}(\alpha_j),\,\log \alpha_i,\, 1\,\}, \quad j=1,2.
\]
Then, for any positive integers $b_1$ and $b_2,$
\[
\log\big|\,b_2 \log {\alpha_2} - b_1 \log {\alpha_1}\,\big|>-25.2\,H_1\,H_2\,\Bigr(\max \{\, \log b'+0.38,\,10\,\}\Bigr)^2
\]
with $b'=b_1/H_2+b_2/H_1.$
\end{prop}

We apply Proposition \ref{prop-comp} with the following parameters:
\[
\alpha_1=c, \ \alpha_2=d, \ b_1=z, \ b_2=w, \ H_1=\log c, \ H_2= \log d.
\]
We have
\begin{equation}\label{ineq-loglam-low}
\log \varLambda > -25.2 \,\log c\,\log d\, \biggl(\,\max \biggl \{ \log \Bigr( \frac{z}{\log d}+\frac{w}{\log c} \Bigr) +0.38, \,10\biggl\} \,\biggl)^2.
\end{equation}
Since $\frac{z}{\log d}<\frac{w}{\log c}$ as $\varLambda>0$, inequalities \eqref{ineq-loglam-upp} and \eqref{ineq-loglam-low} together yield
\[
s <50.4\,\Bigr( \max\bigr\{\! \log (2s) +0.38,\,10 \bigr\} \Bigr)^2
+\frac{2\log 6}{\log c\,\log d}
\]
with 
\[
s=\frac{w}{\log c}.
\]
This implies that $s<5042$.
Finally, this together with inequality \eqref{ineq-w-low} gives
\[
\frac{1}{\mu_x+\mu_z}\,N < 5042\log (N+1),
\]
implying that $e \le 8$.

\section{Cases with small base numbers}\label{sec-5}

By the contents of previous sections, to complete the proof of Theorem \ref{th-3456}, it suffices to  solve equation \eqref{eq-e} for each $e$ with $e=1,2,\ldots,8$.
The corresponding equations are
\begin{align}
&3^x+4^y+5^z=6^w;\label{3456}\\
&15^x+16^y+17^z=18^w;\label{15161718}\\
&63^x+64^y+65^z=66^w;\\
&255^x+256^y+257^z=258^w;\\
&1023^x+1024^y+1025^z=1026^w;\\
&4095^x+4096^y+4097^z=4098^w;\\
&16383^x+16384^y+16385^z=16386^w;\\ 
&65535^x+65536^y+65537^z=65538^w,
\end{align}
respectively.
For handling these, we use the algorithm of Bert\'ok and Hajdu \cite[Sec.\,3]{BerHa} which is efficient for handling a given purely exponential Diophantine equation.
Indeed, their algorithm frequently succeeds in finding some integer $M$ with the property that reducing the equation in question modulo $M$ reveals that at least one of the unknown exponents of it should be bounded, which often leads us to obtain the non-existence of solutions.
It turns out that the algorithm found appropriate moduli as in Table 1.
\begin{table}[h] \label{table1}
 \begin{center}
\caption{Moduli to \eqref{eq-e} with $1 \le e \le 8$.}
\begin{tabular}{ccc} \hline
   equation & modulus & output    \\ \hline
   (5.1) & $2^4 \cdot 3^3 \cdot 7\cdot13\cdot73$ & $w \le 3$ \\
   (5.2) & $2^2\cdot7\cdot13\cdot19\cdot37\cdot73$ &  no solution\\
   (5.3) & $2^2\cdot13\cdot37\cdot73$ &  no solution \\
   (5.4) &$2^2\cdot7\cdot13\cdot19\cdot37\cdot73$ &   no solution \\
   (5.5) &$2^2\cdot7\cdot13\cdot19\cdot37\cdot73$&   no solution \\
   (5.6) &$2^2\cdot37\cdot73\cdot163\cdot433\cdot1297$&   no solution \\
   (5.7) &  $2^2\cdot7\cdot13\cdot19\cdot37$ & no solution \\      
   (5.8) &$2^2\cdot7\cdot13\cdot19\cdot37\cdot73$  &  no solution \\ \hline
\end{tabular}
 \end{center}
\end{table}

From this table we can easily conclude that \eqref{3456}, which is already solved in \cite{AlFo2}, has only the solutions $(x,y,z,w)=(3,1,1,2),(3,3,3,3)$ and that each of the other equations has no solution already modulo some integer $M$.
Below, we shall briefly explain why the choice of $M=2^4 \cdot 3^3 \cdot 7\cdot13\cdot73$ is appropriate for \eqref{3456} as in the above table.

According to some of the values of unknown exponents, we choose moduli to find an upper bound for $w$, as in Table 2.
\begin{table}[h] \label{table2}
 \begin{center}
\caption{Moduli to \eqref{3456} with output.}
\begin{tabular}{ccc} \hline
   exponent restriction & modulus & output    \\ \hline
   $x=2$ & $2^2$  & $w \le 1$ \\
   $x \ne 2,\,y=1$ & $2^3\cdot3$ & $w \le 2$ \\
   $x=1,\,y \ge 2$ & $2^3\cdot3$ & $w \le 2$ \\     
   $x \ge 3,\,y \ge 2$ &$2^4 \cdot 3^3 \cdot 7\cdot13\cdot73$  &  $w \le 3$ \\ \hline
\end{tabular}
 \end{center}
\end{table}

We will use the following notation: 
\[
(x_1,x_2,\ldots,x_k) \equiv (y_1,y_2,\ldots,y_k) \mod{(m_1,m_2,\ldots,m_k)}
\]
which stands for $x_i \equiv y_i \pmod{m_i}$ for $i=1,2,\ldots,k$.

\bigskip {\it Case where $x=2$ and $y=1.$} 
We have 
\[
13+5^z=6^w.
\]
Reducing this modulo $4$ gives that $2 \equiv 2^w \pmod{4}$, so that $w \le 1$.

\bigskip {\it Case where $x \ne 2$ and $y=1.$}
Reducing the equation modulo $3$ gives that $1+(-1)^z \equiv 0 \pmod{3}$, so that $z \equiv 1 \pmod{2}$.
Also, reducing it modulo $8$ gives that $3^x+4+5^z \equiv 6^w \pmod{8}$, so that $w \le 2$ or $3^x+4+5^z \equiv 0 \pmod{8}$.
The congruence in the latter case implies that
\[
(x,z) \equiv (1,0) \mod{(2,2)}.
\]
Putting all found restrictions together implies that $w \le 2$.

\bigskip{\it Case where $x=2$ and $y \ge 2$}.
We have
\[
9+4^y+5^z=6^w.
\]
This implies that $w \le 1$ in the same way as that in case where $(x,y)=(2,1)$.

\bigskip {\it Case where $x=1$ and $y \ge 2$}.
We have
\[
3+4^y+5^z=6^w.
\]
Reducing this equation modulo $3^2$ implies that $w \le 1$ or $3+4^y+5^z \equiv 0 \pmod{9}$.
The congruence in the latter case implies that
\begin{equation} \label{x1yge2-cond-mod9}
(y,z) \equiv (0,1),(1,5),(2,3) \mod{(3,6)}.
\end{equation}
Also, considering it with modulus $7$ shows that $3+4^y+5^z \equiv (-1)^w \pmod{7}$, which implies that
\begin{equation} \label{x1yge2-cond-mod7}
\begin{split}
(y,z,w) \equiv\, & (0,2,1),(1,0,0),(2,5,0),\\
&(0,4,1),(1,3,1),(2,0,1) \mod{(3,6,2)}.
\end{split}
\end{equation}
Putting restrictions \eqref{x1yge2-cond-mod9} and \eqref{x1yge2-cond-mod7} together implies that there is no suitable pair $(y,z)$.
Thus, $w \le 1$.

\bigskip

{\it Case where $x \ge 3$ and $y \ge 2$}.
Here we choose 5 moduli $M$ in turn. 
First, we set $M=2^4$.
Then $w \le 3$ or $3^x+5^z \equiv 0 \pmod{16}$. 
The congruence in the latter case implies that
\begin{equation} \label{xge3yge2-cond-mod16}
(x,z) \equiv (1,3),(3,1) \mod{(4,4)}.
\end{equation}
Next, we set $M=7$, and $3^x+4^y+5^z \equiv (-1)^w \pmod{7}$.
This implies that
\begin{equation} \label{xge3yge2-cond-mod7}
\begin{split}
(x,y,z,w)  \equiv \,&( 0, 0, 2, 1 ),    ( 0, 0, 3, 0 ),    ( 0, 1, 0, 1 ),    ( 0, 1, 5, 0 ),    ( 0, 2, 1, 0 ),\\
   & ( 0, 2, 5, 1 ),    ( 1, 0, 2, 0 ),    ( 1, 0, 4, 1 ),    ( 1, 1, 0, 0 ),    ( 1, 1, 3, 1 ),\\
    &( 1, 2, 0, 1 ),    ( 1, 2, 5, 0 ),    ( 2, 0, 1, 0 ),    ( 2, 0, 5, 1 ),   ( 2, 1, 4, 0 ),\\
   & ( 2, 2, 2, 0 ),    ( 2, 2, 4, 1 ),    ( 3, 0, 0, 0 ),    ( 3, 0, 3, 1 ),    ( 3, 1, 1, 0 ),\\
&    ( 3, 1, 5, 1 ),    ( 3, 2, 1, 1 ),    ( 4, 0, 0, 1 ),    ( 4, 0, 5, 0 ),    ( 4, 1, 1, 1 ),\\
 &   ( 4, 2, 4, 0 ),     ( 5, 0, 4, 0 ),    ( 5, 1, 2, 1 ),    ( 5, 1, 3, 0 ),    ( 5, 2, 0, 0 ),\\
&    ( 5, 2, 3, 1 )
\mod{(6,3,6,2)}.
\end{split}
\end{equation}
Putting restrictions \eqref{xge3yge2-cond-mod16} and \eqref{xge3yge2-cond-mod7} together implies that $w \le 3$ or 
\begin{equation} \label{xge3yge2-cond-H1}
(x,y,z,w) \in H_1 \mod{(12,3,12,2)}
\end{equation}
for some set $H_1$ composed of 11 elements. 
We continue this process with $M=3^3, 13,73$ in turn.
The derived restriction on $(y,z)$ modulo $(9,18)$ by choice $M=3^3$ together with \eqref{xge3yge2-cond-H1} implies that $(x,y,z,w) \in H_2 \pmod{(12,9,36,2)}$ for some set $H_2$ of 18 elements.
Further, taking $M=13$ leads to $(x,y,z,w) \in H_3 \pmod{(12,18,36,12)}$ with $|H_3|=15$.
Finally, taking $M=73$ leads to the non-existence of $(x,y,z,w)$, so that $w \le 3$.

\bigskip

\subsection*{Acknowledgements}
The author is indebted to Lajos Hajdu for his kind help to use the algorithm of paper \cite{BerHa} in Section \ref{sec-5}.

\end{document}